\documentclass[12pt,aps,prb, preprint,reqno]{amsart}
\usepackage{amsmath,amssymb}
\usepackage{amsfonts}
\usepackage{amsthm}
\usepackage{yfonts}
\usepackage{hyperref}
\usepackage{mathrsfs}
\usepackage{latexsym}
\usepackage{graphicx}
\usepackage{epsfig}
\usepackage{enumitem}
\usepackage{xcolor}
\usepackage{txfonts}
\usepackage{cite}

\newtheorem{theorem}{Theorem}[section]
\newtheorem{lemma}[theorem]{Lemma}
\newtheorem{proposition}[theorem]{Proposition}
\newtheorem{definition}[theorem]{Definition}
\newtheorem{corollary}[theorem]{Corollary}

\newtheorem{example}{Example}[section] 

\newtheorem{assumption}{Assumptions}

\newtheorem{prop}{Proposition}[section]
\newtheorem{remark}[prop]{Remark}

\makeatletter \@addtoreset{equation}{section} \makeatother

\def\ddt{\frac{d}{dt}}

\def\RR{{\mathrm R}}
\def \2R{{\hat{\RR}}}

\def\Rc{{\mathrm {Rc}}}

\def\SS{{\mathrm S}}

\def\He{\mathrm {Hess}}

\def\id{\mathrm{Id}}

\def\xi{\partial_{x_i}}

\def\tr{\text{tr}}

\usepackage[margin=1.19 in]{geometry}

\begin{document}
\title{On Ricci Solitons with Isoparametric Potential Functions}
\author[Hung Tran]{Hung Tran}
\address{Department of Mathematics and Statistics,
	Texas Tech University, Lubbock, TX 79413}
\email{hung.tran@ttu.edu} 


\author[Kazuo Yamazaki]{Kazuo Yamazaki}
\address{Department of Mathematics, University of Nebraska, Lincoln, 243 Avery Hall, PO Box, 880130, Lincoln, NE 68588-0130, U.S.A.; Phone: 402-473-3731; Fax: 402-472-8466}
\email{kyamazaki2@nebraska.edu}



\date{}

\allowdisplaybreaks
\begin{abstract} This paper studies a complete gradient Ricci soliton with an isoparametric potential function. Our first theorem asserts that, for the steady case, there is a critical level set of codimension greater than one. This is consistent with construction of cohomogeneity one models with singular orbits. There is a partial result for the shrinking case. We also study a particular ansatz of popular interest and obtain asymptotic behaviors.  
\end{abstract}
\maketitle
\tableofcontents

\section{Introduction}
A gradient Ricci soliton (GRS) arises naturally in the theory of Ricci flows and its celebrated applications including \cite{H3, HPCO, Hsurvey, perelman1, perelman2, perelman3, bs091, bs072, bohmwilking}; see \cite{caotran6survey} for a recent survey. Indeed, a GRS $(M^n, g, f, \lambda)$ is a Riemannian smooth $n$-dimensional manifold $M$ with metric $g$, potential function $f$, and a constant $\lambda$ such that, for $\Rc$ denoting the Ricci curvature,  
\begin{equation}\label{GRS}
	\Rc+\text{Hess}{f}=\lambda g. 
\end{equation}
Corresponding to the sign of $\lambda$, a GRS is called shrinking $(\lambda>0)$, steady ($\lambda=0$), or expanding $(\lambda<0 )$. In addition, a GRS is also a generalization of an Einstein metric, a topic of great independent interest.

Many known non-trival explicit examples are constructed with cohomogeneity one symmetry \cite{caohd96, CV96, caohd97limits, FIK03Kahler, koi90, PTV99quasi, dw11, buttsworth2021}. It is curious to observe that, in all such constructions, there is always a singular orbit. Our first result will show that property is indeed necessary for a large family of GRSs. 

A smooth function on a Riemannian manifold is called transnormal if the norm of its gradient is constant on each level set; equivalently, its integral curves are reparamertrized geodesics. If, additionally, the mean curvature of each level set is constant, then the function is called isoparametric. The study of those functions in space forms was motivated by questions in geometric optics and initial contributions were given in \cite{som1918, cartan38, lv37}. The classification in an ambient round sphere, formulated as Question 34 in S.T. Yau's list \cite{Yauopen93}, is remarkably deep. It has attracted enormous interest and important developments are given by, for example, \cite{nomizu73, Munzner80iso, Abresch83foursix, CCJ07four, chi20four}; see \cite{chi20survey} for a recent survey. Certain aspects of the theory can be extended to a general Riemannian manifold \cite{wang87iso, gt13exotic, miyaoka13, QT15, derd21}.

In the context of a GRS, it was observed in \cite{tran23contact} that a transnormal potential function is automatically isoparametric. Furthermore, all aforementioned cohomogeneity-one examples have isoparametric potential functions. An equivalent condition is considered in \cite{kim17harmonic} leading to a classification for GRS with harmonic Weyl curvature. It also arises as a consequence of constant scalar curvature \cite{FG16csc, CZ2021rigidity}. Furthermore, in \cite[Section 3]{CDSexpandshriking19}, the authors discovered that an appropriate resealing of the potential function, on an expanding KGRS, converges to a transnormal function with respect to a suitable metric. Additionally, the first author was able to classify all complete K\"{a}hler GRS surfaces with such a potential function \cite{tran23contact}.  

In this note, we study a real GRS with an isoparametric potential function and obtain the following results. Here a singular foil means a level set corresponding to a critical value and of co-dimension greater than one. 
\begin{theorem}
	\label{main1}
	Let $(M, g, f)$ be a non-trivial complete steady GRS with a transnormal potential function $f$. Then $f$ has a unique singular foil corresponding to its global maximum.
\end{theorem}
\noindent For the shrinking case, we have the following. 
\begin{theorem}
	\label{main2}
	Let $(M, g, f)$ be a complete shrinking GRS with a transnormal potential function $f$. Then either $f$ has a singular foil or the manifold is diffeomorphic to $\mathbb{R} \times P$ for some compact $P$. 
\end{theorem}

\begin{corollary}
	\label{coro1}
	Let $(M, g, f)$ be a compact shrinking GRS with a transnormal potential function $f$. Then there must be exactly two singular foils of $f$ corresponding to its global minimum and maximum.  
\end{corollary}
\begin{remark} An example for Corollary \ref{coro1} is the construction of a shrinking (K\"{a}hler) GRS on the blow-up of $\mathbb{CP}_n$ at one point \cite{koi90,caohd96}. 
\end{remark}

In a general high dimension, a complete classification of GRS is currently out of reach. Even the cohomogeneity-one case poses a significant challenge which is two-fold: the homogeneous principal orbit and the associated ODE system. In this note, we study a particular ansatz, generalizing the cohomogeneity one setup, and focus on the technicality of the latter. 
  
 \textit{ \textbf{An Ansatz:} Let ${N}^{n-1}(k)$ be a K\"{a}hler-Einstein manifold with $\Rc_N=k \id$, $I$ an interval, and functions $H, F: I\mapsto \mathbb{R}^+$. $(P, g_t)$ is a Riemann submersion of a line or circle bundle with coordinate $z$ over $({N}, F^2 g_N)$ and a bundle projection $\pi: P\mapsto \mathbb{N}$. $\eta$ is the one-form dual of $\partial_z$ such that $d\eta=q\pi^\ast \omega_{\mathbb{N}}$ for $q\in \{0, 1\}$. The metric on $I\times P$ is given by} 
 \begin{equation}\label{ansatz}
 	g = dt^2+g_t=dt^2+ H^2(t) \eta\otimes \eta + F^2(t)\pi^\ast g_\mathbb{N}.
 \end{equation}

Let $(M, g, f, \lambda)$ be a complete steady GRS given by ansatz \eqref{ansatz} and $f=f(t)$. By Theorem \ref{main1}, there is a singular foil and, without loss of generality (w.l.o.g), we assume it occurs at $t=0$. We show that such a GRS possesses similar properties to a K\"{a}hler one. 
\begin{theorem} 
	\label{main3}
	Let $(M, g, f, \lambda)$ be a complete steady GRS with the metric given by ansatz \eqref{ansatz} for $q=1$ and $f=f(t)$. Then
	\begin{enumerate}
		\item $H(0)=0$ and $H$ is increasing;
		\item $F$ is either increasing or it has a global minimum. Furthermore, 
		\[\lim_{t\rightarrow \infty}F =\infty. \]
	\end{enumerate}
\end{theorem}
\begin{remark} In explicit construction of K\"{a}hler steady GRS, $H$ is also observed to be bounded above \cite{dw11, tran23kahler}. 
\end{remark} 

\begin{remark}
	This ansatz is of popular interest. If ${N}=\mathbb{CP}^{n-1}$ and $q=1$ then $P=\mathbb{S}^{2n-1}/\mathbb{Z}_k$ and one recovers the Hopf fibration. As shown in \cite{tran23kahler}, a non-trivial K\"{a}hler GRS of maximal symmetry is constructed by the above ansatz where $(N, g_N)$ is simply-connected of constant holomorphic sectional curvature. Furthermore, there are several efforts to study a Ricci flow on such an ansatz, under various topological and analytical conditions, and determine the corresponding singularity models \cite{appleton2023, di20}.    
\end{remark}



\begin{remark}
	The case $q=0$ corresponds to each level set being a product metric and the manifold a double-warped product. One explicit construction is given in \cite{Ivey94}; also Dancer-Wang study multiple warped products to construct non-K\"{a}hler steady and expanding examples in \cite{dw09, DW09expanding} respectively. 
	The case $q=1$ corresponds to each level set being a deformed Sasakian strucutre. Dancer-Wang \cite{dw11} studies a generalization of this ansatz, with multiple $F$'s and $q$'s, to construct K\"{a}hler GRS. 
\end{remark}

We give various necessary preliminaries in Section \ref{Section 2, Preliminaries}, prove Theorems \ref{main1}-\ref{main2} in Section \ref{Section 3, Singular Foils}, and prove Theorem \ref{main3} in Section \ref{Section 4, Estimates}. 

\section{Preliminaries}\label{Section 2, Preliminaries}

\subsection{Isoparametric Functions}
Let $(M, g)$ be a Riemannian manifold and $f: M\mapsto \mathbb{R}$ be a smooth function.
\begin{definition}\label{Definition 2.1, draft}
	$f$ is called \textit{transnormal} if there is a smooth function $b: \mathbb{R}\mapsto \mathbb{R}$ such that $|\nabla f|^2= b(f).$ A transnormal function $f$ is called \textit{isoparametric} if there is a continuous function $a:\mathbb{R}\mapsto \mathbb{R}$ such that $\Delta f= a(f).$ A component of a level set is called a foil if it has codimension one and singular foil (focal submanifold) if it has codimension bigger than one. 
\end{definition}

By Gauss lemma, a Riemannian manifold with a transnormal function is locally foliated by equidistant regular and singular foils. The isoparametric condition further implies that each regular foil has constant mean curvature. A transnormal function induces a transnormal system so that any geodesic of $M$ cuts the foils perpendicularly at none or all of its points \cite{bolton73}. In particular, the exponential map of the normal bundle to a foil will generate geodesics which cut the foils perpendicularly at all of its points. If $\gamma(t)$ is such a geodesic with parametrized arc length, we have $\gamma'(t)=\pm\frac{\nabla f}{|\nabla f|}$ whenever $|\nabla f|\neq 0$. They are called $f$-segments accordingly. There is a classification of complete connected Riemannian manifolds with a transnormal function up to diffeomorphism \cite[Theorem 1.1]{miyaoka13}.



\begin{theorem}
	\label{structuretrans}
	Let $M$ be a complete connected Riemannian manifold which admits a transnormal function $f$. Then either one of the followings holds:
	\begin{enumerate}
		\item $M$ is diffeomorphic to a vector bundle over a submanifold $Q$ of $M$.
		\item $M$ is diffeomorphic to a union of two disk bundles over two submanifolds $Q$ and $Q'$ of $M$ where $Q$ and/or $Q'$ might be hypersurface(s). 
	\end{enumerate}
\end{theorem}
\begin{remark}
	In particular, a Riemannian manifold with a transnormal function admits at most two singular foils. Each corresponds to a global minimum or maximum of the function. 
\end{remark}


\subsection{Curvature Computation}
\label{curvaturecomp}

Let $(M, g, f)$ be a Riemannian manifold with metric $g$ and a transnormal function $f$. Wherever $\nabla f\neq \vec{0}$, nearby level sets of $f$ are diffeomorphic to each other via $f$-segments. Thus, the manifold locally is diffeomorphic to $P\times I$ where $P$ is a regular level set and $I$ is an open interval. Furthermore, we let $g_t$ to be the restriction of $g$ to $P\times \{t\}$ for $t\in I$. The transnormality of $f$ implies that $g$ then takes the simple form:
\begin{align}
	\label{bundlelikemtric}
	g &= dt^2 + g_t.
\end{align}
Via the identification of tangent bundles, $g_t$ can be considered as a one-parameter family of metrics on $P$. 

Following \cite[Remark 2.18]{dw11}, we recall the shape operator $L(X):= \nabla_X \nu$ for $\nu:=\frac{\partial}{\partial t}$. One observes that
\[\partial_t g= 2g_t\circ L_t .\] 
Thus the Ricci curvature of $(M, {g})$ is totally determined by the geometry of the shape operator and how it evolves. Precisely, Gauss, Codazzi, and Riccati equations imply the following computation,  for tangential vector fields $X, Y$ of $P$, 
\begin{align*}
	\Rc(X, Y) &=\Rc_t(X, Y)-\tr(L)g_t(LX, Y)-g_t(\dot{L}(X), Y),\\
	\Rc(X, \nu) &=-\nabla_X \tr(L_t)-g_t(\delta L, X),\\ 
	\Rc(\nu, \nu) &= -\text{tr}(\dot{L})-\text{tr}(L^2).
\end{align*}
Here $\delta L=\sum_{i}\nabla_{e_i}L(e_i)$ for an orthonormal basis $\{e_{i}\}_{i}$ and $\tr{T}=\tr_{g_t}T_t$. Furthermore,   
\begin{align*}
	\He{f}(X, \nu) &=0,\\
	\He{f}(\nu, \nu)&= f'',\\
	\He{f}(X, Y)&= f' g_t(L_t(X), Y). 
\end{align*}

Next, we consider a GRS $(M, g, f)$ with $\He f+ \Rc=\lambda g$. The following identities are well-known \cite{chow2023}, for $\SS$ denoting the scalar curvature,
\begin{align}
	\label{deltaS}
	\SS + \triangle f &= n\lambda;\\
	\label{rcandf}
	\Rc(\nabla{f})&=\frac{1}{2}\nabla{\SS};\\
	\label{nablafandS}
	\SS+|\nabla f|^2-2\lambda f &= C_1;\\
	\label{lapS}
	\triangle\SS+2|\text{Rc}|^2 &= \left\langle{\nabla f,\nabla \SS}\right\rangle+2\lambda \SS.
\end{align}
If the potential function $f$ is transnormal, then \eqref{GRS} is reduced to a system:
\begin{align}
	0 &=-(\delta L)-\nabla \tr{L},\nonumber \\
	\label{reducedsolitionsystem}
	\lambda &=-\text{tr}(\dot{L})-\text{tr}(L^2)+f'',\\
	\lambda g(X, Y) &=\Rc(X, Y)-(\tr{L}) g(L X, Y)-g(\dot{L}(X), Y) +f' g(LX, Y). \nonumber
\end{align}
Furthermore, (\ref{deltaS}) becomes:
\begin{align}
	\label{scalar1}
	{\SS} &=\lambda n-{\Delta} f=\lambda n- f''-f'\tr{L}
\end{align} 
As observed by Dancer-Wang \cite{dw11}, it is useful to consider the conservation law:
\begin{equation}\label{conservationlaw}
	f''+(\tr{L}) f'-(f')^2+2\lambda f=C=\lambda n-C_1.
\end{equation}

At a singular foil $Q$, let $NQ$ be the normal bundle of dimension $k>1$. At $Q$, one no longer has  equation (\ref{bundlelikemtric}) and approaching $Q$,  $\tr{L}$ diverges toward $\pm \infty$. However, by continuity, $f''$ is still well-defined and, indeed, for any unit vector $V\in NQ$, $f''= \He f(V, V).$ Thus, (\ref{scalar1}) is adjusted as follows:
\begin{align}
	\label{singularscalar}
	\SS &= \lambda n- k f''= C_1+2\lambda f.	
\end{align}

\section{The Existence of Singular Foils}\label{Section 3, Singular Foils}
In this section, we consider a GRS with a potential transnormal function $f$ and show that, if $\lambda= 0$, there must be a singular foil. The shrinking case is relatively delicate and we only obtain a partial result. We start with a few general observations. 
\begin{lemma}
	\label{noentirefunction}
	A $C^1$-function $u:\mathbb{R}\rightarrow \mathbb{R}$ such that $u'\leq -c u^2$ for some positive constant $c$ must be equivalently zero. 
\end{lemma}
\begin{proof}
	Suppose $u(a)\neq 0$ for some $a\in \mathbb{R}$. If $u(a) > 0$, then $u(b) > 0$ for any $b< a$ due to the hypothesis so that 
	\begin{align*}
	\int_{b}^a \frac{u'(x)}{u^2(x)} dx&=-\frac{1}{u(a)}+\frac{1}{u(b)}
	\end{align*}
is well-defined. For $b \in (a - \frac{1}{C u(a)}, a)$, the hypothesis shows that $u(b)\geq \frac{1}{\frac{1}{u(a)}-c(a-b)}$. This expression is approaching positive infinity as $b\searrow a-\frac{1}{cu(a)}$, a contradiction. If  $u(a)<0$, then $u(b) < 0$ for any $b>a$ so that 
\begin{align*}
\int_{a}^{b} \frac{u'(x)}{u^{2}(x)} dx = \int_{a}^{b} \frac{d}{dx} (-u^{-1})(x) dx = -\frac{1}{u(b)} + \frac{1}{u(a)} 
\end{align*}
is well-defined. For $b \in (a, a- \frac{1}{C u(a)}, a)$, the hypothesis shows that $u(b)\leq \frac{1}{\frac{1}{u(a)}+c(b-a)}$. This expression is approaching negative infinity as $b\nearrow a-\frac{1}{cu(a)}$, a contradiction. Thus, $u\equiv 0$. 
\end{proof}

\begin{lemma}
	\label{f'decreasing}
	If $f''\leq \lambda$ on $\mathbb{R}$, then $f''\equiv \lambda$ and $\SS=\lambda (n-1)$. 
\end{lemma}
\begin{proof}
The equation (\ref{reducedsolitionsystem}) and our assumption imply that
\begin{equation}\label{est 1 draft}
\ddt \tr(L)= \tr{\ddt{L}}= f''-\tr(L^2)-\lambda \leq -\tr(L^2).
\end{equation}
Since $L$, representing the shape operator, is symmetric, we apply the Cauchy-Schwarz inequality to obtain:
\[\ddt \tr(L) \leq -\frac{1}{n}(\tr(L))^2.\]
Thus, Lemma \ref{noentirefunction} is applicable and implies $\tr(L)\equiv 0$. Consequently, $\frac{d}{dt} \tr(L) \equiv 0$; applying this to \eqref{est 1 draft} shows $f''-\lambda = \tr (L^{2}) \geq 0$; thus,  $f''\equiv \lambda$ . The claim that $\SS=\lambda (n-1)$ now follows from \eqref{scalar1}. 
\end{proof}

\begin{lemma}\label{casesf'}
If $\lambda\leq 0$ and $\SS>0$, then $f'$ is either strictly monotonic or $f'$ has a unique critical point. In the latter case, either $f'>0$ and the critical point is a global minimum or $f'<0$ and the critical point is a global maximum.  
\end{lemma}
\begin{proof}
	Differentiating (\ref{conservationlaw}) yields
\begin{equation}\label{est 2, draft}
f'''+(\tr(L))'f'+\tr(L)f''-2f'f''+2\lambda f'=0. 
\end{equation} 	
Suppose $f'$ has a critical point at $a$ then $f''(a)=0$ and, using \eqref{reducedsolitionsystem}, we have at point $a$,
\begin{align*} f'''=-(\tr(L))'f'-2\lambda f' =(\lambda+\tr(L^2)-f''-2\lambda) f' =(\tr(L^2)-\lambda)f'.
\end{align*}
Since $\SS>0$ and $\lambda\leq 0$ by hypothesis, \eqref{scalar1} and $f''(a) = 0$ imply that 
\[f'(a)\neq 0 \neq \tr(L)(a).\]
Since $L$ is symmetric, $\tr(L)(a)\neq 0$ implies $0<\tr(L^2)(a)$. Therefore, $0<\tr(L^2)-\lambda$ as $\lambda\leq 0$ which implies that $f'''(a) \neq 0$. Consequently, the critical point of $f'$ is either a global minimum (and $f'>0$ everywhere on its domain) or a global maximum (and $f'<0$ everywhere on its domain).    
\end{proof}


\subsection{The Steady Case}
\begin{lemma}\label{f''monotonic}
If $\lambda= 0$, $f'<0$, and $f''>0$ on an interval $I$, then $f''$ is strictly decreasing on $I$.
\end{lemma}
\begin{proof}
The equation (\ref{scalar1}) and $\SS>0$ imply that $\tr(L)>0$ on $I$. Applying (\ref{reducedsolitionsystem}) to \eqref{est 2, draft} yields 
\begin{align*}
0 =f''' -\tr(L^2)f'+(\tr(L)-f')f''.
\end{align*} 
Thus, $f'''<0$ on $I$.	
\end{proof}

\begin{proposition}\label{existencecritical}
Let $(M, g, f)$ be a non-trivial complete steady GRS with a transnormal potential function $f$. Then, there is a point $x\in M$ such that $\nabla f\mid_x =\vec{0}$.   
\end{proposition}
\begin{proof}
	We proceed by contradiction. Suppose that $\nabla f\neq \overrightarrow{0}$ everywhere. Pick a point $x$ and consider the geodesic with initial data $\gamma(0)=0, \ddt\gamma\mid_{t=0}=\frac{\nabla f}{|\nabla f|}$. By our hypothesis of $f$ being transnormal, $\gamma$ meets each level set perpendicularly and exactly once. By the completeness of the metric, the manifold is diffeomorphic to $P\times \mathbb{R}$ where $P$ is a hypersurface. Thus, the domain of each function $f, f', f'', \tr{L},$ and $\tr{L^2}$ is the entire real line $\mathbb{R}$. By choosing a direction on $\mathbb{R}$, we can assume $f'<0 $ everywhere.
	
By \cite{zhang09completeness}, a complete steady GRS must have $\SS\geq 0$ and $\SS=0$ at a point if and only if the soliton is trivial. Thus, our assumption guarantees $\SS>0$. Then, Lemma \ref{casesf'} applies to deduce that either $f'$ is strictly monotonic or $f'$ has exactly one unique critical point. We consider each case.
	
	\textbf{Case 1:}	 $f'$ is strictly monotonic and $f''<0$ everywhere on $\mathbb{R}$. 	By Lemma \ref{f'decreasing}, $\SS=0$, a contradiction. 
	
	\textbf{Case 2:} $f'$ is strictly monotonic and $f''>0$ everywhere on $\mathbb{R}$. By  Lemma \ref{f''monotonic}, $f''$ is strictly decreasing on $\mathbb{R}$ and, thus, $0< \lim_{t\rightarrow -\infty}f''$.  However, (\ref{nablafandS}) implies that $f'$ is strictly monotonic and bounded. Thus, it must converge to a finite constant and, consequently, $f''$ converges to zero, a contradiction. 
	
	\textbf{Case 3:} $f'$ has a unique critical point, at $a\in \mathbb{R}$, whose value $f'(a)$ is a global maximum by Lemma \ref{casesf'} and the fact that $f'<0$. Then, there is non compact interval $(-\infty, a)$ on which $f''>0$. We apply the same argument as in case 2 to obtain a contradiction.
	
	  Thus, there must be a point $x\in M$ such that $\nabla f\mid_x =\vec{0}$.
\end{proof}
\begin{proposition}
	\label{criticalimpliessingular}
	Let $(M, g, f)$ be a non-trivial complete steady GRS with a transnormal potential function $f$. By Proposition \ref{existencecritical}, $\nabla f(x)=\vec{0}$ at some point $x\in M$. Then the corresponding critical level set containing $x$ is a singular foil.
\end{proposition}

\begin{proof}
By \cite{zhang09completeness}, a complete steady GRS must have $\SS\geq 0$ and $\SS=0$ at a point if and only if the soliton is trivial. Thus, our assumption guarantees that $\SS>0$. By either \eqref{singularscalar} or \ref{scalar1}, $f''(x)<0$ and $f$ achieves a local maximum at $x$. We first show that there is no other critical level set. If there is another separate critical level set, then by the same reasoning, $f$ must achieve another local maximum, a contradiction.
		
Next, we proceed by contradiction. Suppose that the critical level set $P$ is not singular so that it is of codimension one. Let $\gamma(t)$ be an $f$-segment intersecting $P$ at $y=\gamma(0)$. The tangent vector of $\gamma'(0)$ must be a normal vector to $P$. Since $P$ is of codimension one, $\gamma$ must pass through P; that is, locally, $M$ is diffeomorphic to a product of $P$ and a segment of $\gamma$. Thus, $M$ must be globally diffeomorphic to $P\times \mathbb{R}$ or $P\times \mathbb{S}^1$. The latter case is ruled out since there is no other critical point. 
		
		Thus, the domain of each function $f, f', f'', \tr{L},$ and $\tr{L^2}$ is the entire real line $\mathbb{R}$.  By Lemma \ref{casesf'}, $f'$ is decreasing and $f''<0=\lambda$ everywhere. This allows us to apply Lemma \ref{f'decreasing} and conclude that $\SS = \lambda (n-1) = 0$, a contradiction. The proof is complete.
\end{proof}

The proof of Theorem \ref{main1} can be concluded now. 
\begin{proof}[Proof of Theorem \ref{main1}]
It follows from Propositions \ref{existencecritical} and \ref{criticalimpliessingular}.
\end{proof}

We have the following additional corollary. 
\begin{corollary}
Let $(M, g, f)$ be a non-trivial complete steady GRS with a transnormal potential function $f$. Then, the following statements hold: 
\begin{enumerate}
\item $f$ is strictly concave on $(0, \infty)$;
\item The mean curvature $\tr L$ is strictly decreasing and satisfies $0<\tr L\leq \frac{n}{t}$.
\item The scalar curvature is decreasing and tends to zero as $t\rightarrow \infty$. 
\end{enumerate}
\end{corollary}
\begin{proof}
By Theorem \ref{main1}, there must be a singular foil (recall Definition \ref{Definition 2.1, draft}). The results then follow from \cite[Propositions 2.3, 2.4, and 2.5]{BDW15}, respectively. That paper considers the cohomogeneity one case but the proofs of the aforementioned results only use (\ref{reducedsolitionsystem})-(\ref{conservationlaw}) which holds for our more general setup.   
\end{proof}

\subsection{The Shrinking Case}
Here we consider a shrinking GRS $(M, g, f, \lambda)$; that is, $\lambda>0$. W.l.o.g., we can normalize so that $\lambda=\frac{1}{2}$ and rewrite \eqref{nablafandS} as   
\[\SS+|\nabla f|^2=f .\]
First, we recall the following estimate on the potential function $f$ by \cite{HM11comp} (see also \cite{caozhou10}). There exists a point $p\in M$ where $f$ attains its minimum and, furthermore,  
\begin{equation}\label{growthf}
	\frac{1}{4}(d(p, x)-c_1)^2\leq f(x)\leq \frac{1}{4}(d(x, p)+c_2)^2,\end{equation}
for constants $c_1, c_2$ depending only on $n$.

\begin{proof}[Proof of Theorem \ref{main2}]
	By (\ref{growthf}) and the completeness of the metric, each level set of $f$ is compact. Also, since $f$ attains its minimum, the set of critical points is non-empty. Suppose that there is no singular foil. Then each critical level set is a hypersurface. By the same argument as in the proof of Proposition \ref{criticalimpliessingular}, the manifold is diffeomorphic to $P\times \mathbb{R}$ or $P\times \mathbb{S}^1$. We show that the latter case is not possible. If the manifold is diffeomorphic to  $P\times \mathbb{S}^1$, then $f'-\tr(L)$ is periodic. On the other hand, by equation (\ref{reducedsolitionsystem}), 
\[\frac{d}{dt}(f'-\tr{L})=\lambda+ \tr(L^2)=\frac{1}{2}+ \tr(L^2)\geq \frac{1}{2}. \]
Thus, a contradiction. 
\end{proof}




\section{A Particular Ansatz}\label{Section 4, Estimates}
The aim of this Section \ref{Section 4, Estimates} is to prove Theorem \ref{main3} and make various observations through transformations. Here we consider a GRS with the metric given by ansatz \ref{ansatz} and $f=f(t)$. Immediately, the soliton equation is reduced to a system of ODEs.


\begin{lemma}
	\label{solitonsubmer1} Let $(M, g)$ be given as in ansatz \ref{ansatz} and $f=f(t)$ and $n-1=2m$.
The GRS equation becomes 
	\begin{subequations}\label{est 1} 
		\begin{align}
			& \lambda = - \frac{H''}{H} - 2m \frac{F''}{F} + f'', \label{est 1a}\\
			& \lambda = \frac{H^{2}q^{2} }{F^{4}} (2m) - \frac{H''}{H} - 2m \frac{H'F'}{HF} + f' \frac{H'}{H},  \label{est 1b}\\
			& \lambda = \frac{k}{F^{2}} - \frac{H^{2}q^{2}}{F^{4}} 2 - \frac{F''}{F} - (2m-1) (\frac{F'}{F})^{2} - \frac{H'F'}{FH} + f' \frac{F'}{F}. \label{est 1c}
		\end{align}
	\end{subequations}  
\end{lemma}
For a proof, see \cite[Remark 1.5]{tran23kahler}. 

\subsection{Special Solutions}\label{Explicit solution section}
The purpose of this section is to point out that some special case of \eqref{est 1} can admit complete conclusions easily. We restrict ourselves to the following special conditions:  
\begin{assumption}\label{Assumptions 1} 
\indent \begin{itemize}
	\item $\lambda=0= H'=q$
\end{itemize}
\end{assumption}
\noindent Under these assumptions we get from \eqref{est 1} 
\begin{align*}
	0 &=-(2m)\frac{F''}{F}+f''\\
	0 &=\frac{k}{F^2}- \frac{F''}{F}-(2m-1)(\frac{F'}{F})^2 + f'\frac{F'}{F}.
\end{align*}
Following \cite{BDW15, Ivey94}, we set up
\begin{align*}
	\gamma &= -f'+\tr{L}=-f'+ 2m\frac{F'}{F},\\
	\dot{\gamma} &= -\tr{L^2}=-2m \big(\frac{F'}{F}\big)^2 <0,\\
	ds &= \gamma dt.
\end{align*}
We introduce new variables: 
	\begin{equation}	
		X_1 = 0, X_2 =\frac{\sqrt{2m} \dot{F}}{\gamma F}, Y_1 = \frac{1}{\gamma};	Y_2 = \frac{\sqrt{2m}}{\gamma F}. \label{special j}
	\end{equation}
The triple $(X_{2}, Y_{1}, Y_{2})$ satisfies (cf. \cite[Equations (3.5)-(3.6)]{BDW15}). 
\begin{subequations}\label{special}
	\begin{align}
		\frac{d X_2}{d s} &= X_2(X_2^2-1)+ \frac{kY_2^2}{\sqrt{2m}},\label{special a}\\
		\frac{d Y_1}{d s} &= Y_1 X_2^2, \label{special b}\\
		\frac{d Y_2}{d s} &= Y_2\left(X_2^2-\frac{X_2}{\sqrt{2m}}\right). \label{special c}
	\end{align}
\end{subequations}
Evidently, $Y_{1}$ can be solved in terms of $X_{2}$ and thus we only need to focus on \eqref{special a} and \eqref{special c}. We show that $F$ becomes zero in finite time in case $\gamma$ is initially non-positive; then, via a key observation \eqref{key 1}, we extend this result to the case $\gamma$ is initially positive. Hereafter, we write $A \overset{(\cdot)}{=}B$ when the equality is due to $(\cdot)$. 
\begin{proposition}\label{Proposition 2.11} 
	We fix $s_{0} \in \mathbb{R}$ to be the initial time. Suppose that $\gamma(s_{0}) \leq 0$. Let $k \in \{0,1\}$. Take any $s_{0} \in \mathbb{R}$ and then any initial data 
	\begin{equation}\label{special f} 
		X_{2}(s_{0})  > 1 \text{ and } Y_{2}(s_{0}) > 0. 
	\end{equation} 
	Then there exists 
	\begin{equation}
		s^{\ast} \in \left(s_{0}, s_{0} + \frac{1}{(X_{2}(s_{0}) - 1) X_{2}(s_{0})} \right]
	\end{equation} 
	such that 
	\begin{equation}
		Y_{2}(s) \nearrow + \infty \hspace{2mm} \text{ as } \hspace{2mm} s \nearrow s^{\ast}; 
	\end{equation}
	i.e., a finite-time blow-up of $Y_{2}(s)$. Consequently, $F$ vanishes in finite time.
\end{proposition} 

\begin{proof}
	By hypothesis $k \in \{0,1\}$ and thus $\frac{kY_{2}^{2}}{\sqrt{2m}} \geq 0$. Then, because $X_{2}(s_{0}) > 1$ due to \eqref{special f}, it is clear from \eqref{special a} that $\frac{dX_{2}(s)}{ds} \geq 0$ for all $s \geq s_{0}$ so that $X_{2}(s) \geq X_{2}(s_{0}) > 1$ for all $s \geq s_{0}$. Then we can compute 
	\begin{align*}
		\frac{dX_{2}(s)}{ds} \overset{\eqref{special a}}{\geq} X_{2}(s)(X_{2}^{2}(s) -1) \geq X_{2}(s) (X_{2}(s) + 1) (X_{2}(s_{0}) - 1)  \geq X_{2}^{2}(s) (X_{2}(s_{0}) -1).
	\end{align*}
Standard ODE technique leads us to 
\begin{equation}\label{special d} 
X_{2}(s) \geq \frac{X_{2}(s_{0})}{1 - (s- s_{0}) (X_{2}(s_{0}) - 1) X_{2}(s_{0})} \hspace{3mm} \forall\hspace{1mm} s \in \left[s_{0}, \frac{1}{  (X_{2} (s_{0}) - 1) X_{2}(s_{0}) } + s_{0} \right).
\end{equation} 
We can directly solve $Y_{2}$ from \eqref{special c} to see that 
\begin{align}
Y_{2}(s) = Y_{2}(s_{0}) e^{\int_{s_{0}}^{s} X_{2}^{2}(r) - \frac{X_{2}(r)}{\sqrt{2m}} dr}  = Y_{2}(s_{0}) e^{\int_{s_{0}}^{s} X_{2}(r) [X_{2}(r) - \frac{1}{\sqrt{2m}}]dr}.   \label{special e}
\end{align}
Now we rely on 
\begin{enumerate} 
\item $X_{2}(r) - \frac{1}{\sqrt{2m}} \geq 0$ for all $r \geq s_{0}$ and $m \in \mathbb{N}$ because $ X_{2}(r) \geq X_{2}(s_{0}) > 1$, 
\item $X_{2}(r) \overset{\eqref{special d}}{\geq} X_{2}(s_{0})$ for all $r \in [s_{0}, s]$
\end{enumerate} 
and bound \eqref{special e} from below by 
\begin{align*}
Y_{2}(s) \overset{\eqref{special e}}{\geq} Y_{2}(s_{0}) e^{X_{2}(s_{0}) \int_{s_{0}}^{s} X_{2}(r) - \frac{1}{\sqrt{2m}} dr} \overset{\eqref{special d}}{\geq} Y_{2}(s_{0}) e^{X_{2}(s_{0}) [ \int_{s_{0}}^{s} \frac{X_{2}(s_{0})}{1- (r-s_{0}) (X_{2} (s_{0}) - 1) X_{2}(s_{0} )} dr - \frac{s-s_{0}}{\sqrt{2m}}]}
\end{align*}
where 
\begin{align*}
&\int_{s_{0}}^{s} \frac{1}{1- (r-s_{0}) (X_{2}(s_{0}) - 1) X_{2}(s_{0})} dr \\
=& \frac{1}{(X_{2}(s_{0}) - 1) X_{2}(s_{0})}\ln \left( \frac{1}{1- (s-s_{0}) (X_{2}(s_{0}) - 1) X_{2}(s_{0} )} \right)  \nearrow + \infty 
\end{align*} 
as $s \nearrow s_{0} + \frac{1}{ (X_{2}(s_{0}) - 1) X_{2}(s_{0} )}$. This implies the finite-time blow up of $Y_{2}(s)$ at some  $s^{\ast} \leq s_{0} + \frac{1}{ (X_{2}(s_{0}) - 1) X_{2}(s_{0} )}$. Because $\dot{\gamma} < 0$ and $\gamma(s_{0}) \leq 0$ by hypothesis, in consideration of \eqref{special j}, this implies that $F$ vanishes in finite time. 
\end{proof} 

Next, we shall assume that $\gamma > 0$ leading to $Y_{1} > 0$ due to \eqref{special j}. We can compute 
\begin{align}\label{key 1}
	\frac{Y_{2}}{Y_{1}} = \frac{ \sqrt{2m}}{F}.
\end{align}
This makes it clear that to show that $F$ vanishes in finite time, it suffices to show that $\frac{Y_{2}}{Y_{1}}$ blows up in finite time. To achieve this, we see that we can solve for $\frac{Y_{2}}{Y_{1}}$ explicitly as follows. 
\begin{align*}
\frac{d}{ds} \left(\frac{Y_{2}}{Y_{1}} \right) = \frac{\frac{dY_{2}}{ds} Y_{1} - Y_{2} \frac{dY_{1}}{ds}}{Y_{1}^{2}} \overset{\eqref{special}}{=} \frac{ Y_{2} (X_{2}^{2} - \frac{X_{2}}{\sqrt{2m}}) Y_{1} - Y_{2} (Y_{1}  X_{2}^{2})}{Y_{1}^{2}} = - \frac{Y_{2}}{Y_{1}} \frac{X_{2}}{\sqrt{2m}}
\end{align*}
implying that 
\begin{equation}\label{special o}
\left(\frac{Y_{2}}{Y_{1}} \right)(s) = \left(\frac{Y_{2}}{Y_{1}} \right)(s_{0}) e^{-  \frac{1}{\sqrt{2m}}\int_{s_{0}}^{s} X_{2}(r) dr} \hspace{3mm} \forall s \geq s_{0}.  
\end{equation} 
Now a modification of the proof of Proposition \ref{Proposition 2.11} leads to the following result. 

\begin{proposition}\label{Proposition 4.3, draft}
	Let $k \in \{-1, 0\}$. Take any $s_{0} \in \mathbb{R}$ and then any initial data 
	\begin{equation}\label{special l}
		X_{2}(s_{0})  < -1, \hspace{3mm} Y_{2}(s_{0}) > 0, \hspace{3mm} Y_{1}(s_{0}) > 0. 
	\end{equation} 
	Then there exists 
	\begin{equation}
		s^{\ast} \in \left(s_{0}, s_{0} + \frac{1}{(X_{2}(s_{0}) + 1) X_{2}(s_{0})} \right]
	\end{equation} 
	such that 
	\begin{equation}
		\left(\frac{Y_{2}}{Y_{1}} \right)(s) \nearrow + \infty \hspace{2mm} \text{ as } \hspace{2mm} s \nearrow s^{\ast}; 
	\end{equation}
	i.e., a finite-time blow-up of $(\frac{Y_{2}}{Y_{1}})(s)$ and consequently $F$ vanishes in finite time. 
\end{proposition} 

\begin{proof}
	Because $k \in \{-1,0\}$, we have from \eqref{special a} 
	\begin{align}
		\frac{dX_{2}}{ds} = X_{2}(X_{2}^{2} -1) + \frac{ kY_{2}^{2}}{\sqrt{2m}} \leq X_{2}(X_{2}^{2} -1). \label{special m}
	\end{align}
	Moreover, as $X_{2}(s_{0}) \in (-\infty, -1)$, we see that $X_{2}(s_{0})(X_{2}^{2}(s_{0}) - 1) < 0$ implying that 
	\begin{equation*}
		X_{2}(s) \leq X_{2}(s_{0}) \overset{\eqref{special l}}{<} -1 \hspace{3mm} \forall s \geq s_{0}. 
	\end{equation*} 
	Its consequence is the following:  
	\begin{align*}
		\frac{dX_{2}}{ds} \overset{\eqref{special m}}{\leq} X_{2} (X_{2}^{2} -1) = X_{2} (X_{2} - 1)(X_{2} + 1)  \leq X_{2} (X_{2} -1) (X_{2}(s_{0}) + 1). 
	\end{align*}
	where the last inequality used that $X_{2} < -1, (X_{2} -1) < -2$ and thus $X_{2}(X_{2} -1)> 0$. Next, $X_{2} < -1$ and $X_{2}(s_{0}) + 1 < 0$ so that $X_{2} (X_{2}(s_{0}) + 1) > 0$ allowing us to bound by 
	\begin{align*}
		\frac{dX_{2}}{ds} \leq X_{2}^{2} (X_{2}(s_{0}) + 1). 
	\end{align*}
	It follows that 
	\begin{equation}\label{special n}
		X_{2}(s) \leq \frac{X_{2}(s_{0})}{1- (s-s_{0}) (X_{2} (s_{0}) +1) X_{2}(s_{0})} \hspace{2mm} \forall\hspace{1mm} s \in \left[s_{0}, \frac{1}{ (X_{2}(s_{0}) + 1) X_{2} (s_{0})} + s_{0} \right) 
\end{equation} 
Consequently, 
	\begin{align*}
		&\int_{s_{0}}^{s} \frac{X_{2}(s_{0})}{1- (r-s_{0}) (X_{2}(s_{0}) -1) X_{2}(s_{0})} dr \\
		=& X_{2}(s_{0}) \int_{s_{0}}^{s} \frac{1}{1- (r-s_{0}) (X_{2}(s_{0}) - 1) X_{2}(s_{0})} ds \to -\infty \text{ as } s \nearrow  s_{0} + \frac{1}{(X_{2}(s_{0}) + 1) X_{2}(s_{0})}
	\end{align*}
	because $X_{2}(s_{0}) < 0$. This implies by \eqref{special n} that 
	\begin{align*}
		\int_{s_{0}}^{s} X_{2}(r) dr \to -\infty \text{ as } s \nearrow s^{\ast} \text{ for some } s^{\ast} \leq s_{0} + \frac{1}{(X_{2}(s_{0}) + 1) X_{2}(s_{0})}.
\end{align*}
Plugging this into \eqref{special o} shows that 
\begin{align*}
\left(\frac{Y_{2}}{Y_{1}} \right)(s) = \left(\frac{Y_{2}}{Y_{1}} \right) (s_{0}) e^{-\frac{1}{\sqrt{2m}} \int_{s_{0}}^{s} X_{2}(r) dr} \nearrow + \infty \text{ as } s \nearrow s^{\ast}
\end{align*}
for some $s^{\ast} \leq s_{0} + \frac{1}{(X_{2}(s_{0}) + 1) X_{2}(s_{0})}$. 
\end{proof}

\subsection{Transformation}
In the previous section \ref{Explicit solution section}, under Assumptions \ref{Assumptions 1}, we demonstrated the existence of initial data that produces finite-time blow up in Propositions \ref{Proposition 2.11}-\ref{Proposition 4.3, draft}. We wish to relax some of those assumptions and, in this endeavor, we consider the following transformations inspired from \cite[p. 276]{dw11}: 
\begin{equation}\label{est 2} 
ds = Hdt, \hspace{3mm} \alpha(s) \triangleq H^{2}(t), \hspace{3mm} \beta(s) \triangleq F^{2}(t), \hspace{3mm} \varphi(s) \triangleq f(t). 
\end{equation} 
For simplicity we assume 
\begin{equation}\label{simple} 
	s(0) = 0. 
\end{equation}
Denoting $\dot{X} = \partial_{s} X$ while reserving prime for $\frac{d}{dt}$, we compute from \eqref{est 2} that 
\begin{subequations}\label{est 3} 
	\begin{align}
		& \dot{\alpha} = 2H',  \hspace{5mm} \ddot{\alpha} = \frac{2H''}{H}, \label{est 3a} \\
		& \dot{\beta} = \frac{2FF'}{H}, \hspace{5mm} \ddot{\beta} = \frac{2F'^{2} + 2FF''}{H^{2}} - \frac{2FF'H'}{H^{3}}, \label{est 3b} \\
		& \dot{\varphi} = \frac{f'}{H},  \hspace{5mm} \ddot{\varphi} = \frac{f''}{H^{2}} - \frac{f' H'}{H^{3}}.\label{est 3c} 
	\end{align}
\end{subequations}
Applying these transformations to \eqref{est 1} leads us to 
\begin{subequations}\label{est 52}
\begin{align}
\lambda =& - \frac{\ddot{\alpha}}{2} - \frac{m \alpha \ddot{\beta}}{\beta} + \frac{m \alpha (\dot{\beta})^{2}}{2\beta^{2}} - \frac{m \dot{\beta} \dot{\alpha}}{2\beta} + \alpha \ddot{\varphi} + \frac{\dot{\varphi} \dot{\alpha}}{2}, \label{est 52a}\\
\lambda =& \frac{\alpha q^{2} 2m}{\beta^{2}} - \frac{\ddot{\alpha}}{2} - \frac{m \dot{\alpha} \dot{\beta}}{2\beta} + \frac{\dot{\varphi} \dot{\alpha}}{2},\label{est 52b} \\
\lambda =& \frac{k}{\beta} - \frac{2\alpha q^{2}}{\beta^{2}} - \frac{\alpha \ddot{\beta}}{2\beta} + \frac{ \alpha (\dot{\beta})^{2}}{4\beta^{2}} - \frac{ \dot{\beta} \dot{\alpha}}{2\beta} - \frac{(2m-1) (\dot{\beta})^{2} \alpha}{4\beta^{2}} + \frac{\dot{\varphi} \alpha \dot{\beta}}{2\beta}. \label{est 52c}
\end{align}
\end{subequations}
There is enough similarity in \eqref{est 52a} and \eqref{est 52b} that it is beneficial to subtract them and get a simple equation and then keep \eqref{est 52b} (which is simpler than the 1st) so that the equivalent system may be obtained. Thus, we subtract \eqref{est 52b} from \eqref{est 52a} to obtain 
\begin{equation}\label{est 51}
\ddot{\varphi} = \frac{m \ddot{\beta}}{\beta} - \frac{m (\dot{\beta})^{2}}{2\beta^{2}} + \frac{q^{2} 2m}{\beta^{2}}. 
\end{equation}  
\begin{remark}
	Our calculation is consistent with \cite{dw11} modulo a difference in convention regarding $q$. 
\end{remark}

For the rest of this section, we assume $\lambda=0$. 

\begin{lemma}\label{Lemma 4.2}
	Suppose that $\lambda = 0$. Then $(\alpha, \beta, \varphi)$ that solves \eqref{est 52} satisfies  
	\begin{subequations}\label{est 60}
		\begin{align}
			\ddot{\varphi} =& m\Big(\frac{ \ddot{\beta}}{\beta} - \frac{ (\dot{\beta})^{2}}{2\beta^{2}} + \frac{q^{2} 2}{\beta^{2}}\Big), \label{est 60a} \\
			\ddot{\alpha} =& \frac{\alpha q^{2} 4m}{\beta^{2}} - \frac{\dot{\alpha} \dot{\beta}}{\beta} + \dot{\varphi} \dot{\alpha}, \label{est 60b} \\
			\ddot{\beta} =& \frac{2k}{\alpha}- \frac{4q^{2}}{\beta} - \frac{\dot{\beta} \dot{\alpha}}{\alpha} + \dot{\varphi} \dot{\beta}- \frac{ (\dot{\beta})^{2} (1-m)}{\beta}. \label{est 60c} 
		\end{align}
	\end{subequations}
\end{lemma}

\begin{proof}
	\eqref{est 60a} comes directly from \eqref{est 51} while \eqref{est 60b}-\eqref{est 60c} follow from considering $\lambda = 0$ in \eqref{est 52b}-\eqref{est 52c}.  	
\end{proof}

\begin{remark}\label{Remark 4.2, draft}
	\begin{enumerate}
		\item The equation of $\varphi$ in \eqref{est 60a} indicates that it can be solved completely in terms of $\beta$. Thus, we can focus only on \eqref{est 60b}-\eqref{est 60c}. 
		\item This system \eqref{est 60} is ``somewhat'' simple and produced a few nice cancellations. However, this is still very difficult to solve for a general solution. The system \eqref{est 60} is of 2nd order and any differential inequality arguments to derive the desired finite-time extinction seems difficult. Considering that such equations are 2nd order, one idea would to approach it like wave equation and multiply \eqref{est 60b} and \eqref{est 60c} with $\dot{\alpha}$ and $\dot{\beta}$ respectively and add to obtain 
		\begin{align}
			\frac{1}{2} \partial_{s} ( \lvert \dot{\alpha} \rvert^{2} + \lvert \dot{\beta} \rvert^{2}) = \frac{\alpha q^{2} 4 \dot{\alpha}}{\beta^{2}} - \frac{ (\dot{\alpha})^{2} \dot{\beta}}{\beta} + \dot{\varphi} (\dot{\alpha})^{2} + \frac{2k \dot{\beta}}{\alpha} - \frac{4 q^{2} \dot{\beta}}{\beta} - \frac{( \dot{\beta})^{2} \dot{\alpha}}{\alpha} + \dot{\varphi} (\dot{\beta})^{2}; 
		\end{align}
		however, it is not clear how to proceed from this identity. 
		\item We mention that \eqref{est 60} has some advantages in allowing us to deduce a few new explicit solutions. 
	\end{enumerate} 
\end{remark} 

Let us elaborate on Remark \ref{Remark 4.2, draft} (3). First, under the assumptions of $m = 1$ and $q=1$, if we define
\begin{equation}
X= \alpha\beta, \hspace{3mm} A=\dot{\alpha}\beta, \hspace{3mm} B=\alpha\dot{\beta},
\end{equation}\label{est 3, draft}
then \eqref{est 60b} and  \eqref{est 60c} imply that 
\begin{align*}
	\dot{A} &= 4\frac{X}{\beta^2}+ \dot{\varphi}A,~	\dot{B}= 2k-4\frac{X}{\beta^2}+ \dot{\varphi}B,
	{\ddot{X}}= 2k+\dot{\varphi} \dot{X}.	
\end{align*}
Let $u=e^{-\varphi}$ then
\begin{align*}
2ku &=\ddot{X}u-u\dot{\varphi} \dot{X} = \frac{d}{ds}(\dot{X} u),\\
\dot{X} &= e^{\varphi}\left(\int 2k e^{-\varphi} ds \right).
\end{align*}
Now \eqref{est 60a} implies $\ddot{\varphi} \beta^2 = \ddot{\beta}{\beta}- \frac{ (\dot{\beta})^{2}}{2}+2$. Differentiating with respect to $s$ yields $\dddot{\varphi}\beta^2+ 2\beta\ddot{\varphi}\dot{\beta}=\dddot{\beta}\beta$ so that 
\begin{equation}
	\label{eq1}   \dddot{\varphi}\beta+ 2\ddot{\varphi}\dot{\beta}=\dddot{\beta}.
	\end{equation}
We can solve for $\beta$ explicitly as follows. For $s$ defined by \eqref{est 2}, on an interval where $\dot{\beta}\neq 0$, we can write $s$ as a function of $\beta$ and obtain $\ell$ that satisfies 
\begin{align*}
\ddot{\varphi} &= \ddot{\varphi}(s(\beta))=: \ell(\beta).
\end{align*}
Then equation \eqref{est 60a} could be rewritten as, regarding $\beta$ as the variable and $v(\beta)=\dot{\beta}$:
\begin{align*}
\ell(\beta) = \ddot{\varphi} = \frac{\ddot{\beta}}{\beta} - \frac{ (\dot{\beta})^{2}}{2\beta^{2}} + \frac{2}{\beta^{2}}  =\frac{\frac{\partial v}{\partial \beta} v}{\beta}-\frac{v^2}{2\beta^2}+\frac{2}{\beta^2}  = \frac{ 2 v (\frac{\partial v}{\partial \beta}) 2 \beta - v^{2} 2}{(2\beta)^{2}} + \frac{2}{\beta^{2}}= \frac{\partial (\frac{v^2}{2\beta})}{\partial \beta}+\frac{2}{\beta^2}.
\end{align*}
Integrating with respect to $\beta$ yields:
\begin{align*}
	\frac{v^2}{2\beta} = \int  \frac{ \partial ( \frac{v^{2}}{2\beta})}{\partial \beta} d \beta = \int [ \ell (\beta) - 2 \beta^{-2} ] d \beta = \int \ell(\beta) d\beta+ \frac{1}{\beta}.
\end{align*}
Reversing the change of variables then yields:
\begin{align*}
	\frac{\dot{\beta}^2}{2} = \frac{v^{2}}{2} = 1+\beta \int \ell(\beta) d\beta= 1+\beta\int \ddot{\varphi}\dot{\beta} ds.
\end{align*}
This equation or its sibling \eqref{eq1} provides a recipe to construct explicit local solutions. To aim for any global solution, the known examples suggest that we look for $\lim_{s\rightarrow \infty}\ddot{\varphi}=0$. 

The simplest case is to consider $\ddot{\varphi}\equiv 0$. Then $\beta$ must be quadratic, i.e.,
\begin{align*}
	\beta = c_1 s^2+c_2 s+c_3 \hspace{2mm} \text{ and by (\ref{est 60a})} \hspace{2mm} 4(c_1c_3+1) &= c_2^2.
\end{align*}
Also, $\dot{\varphi}=c$ and $\varphi= cs+d$. It follows immediately that if $c_1\neq 0$, then the equation $c_1 s^2+c_2 s+c_3=0$ has two real roots:
\begin{align*}
\frac{-c_{2} \pm \sqrt{c_{2}^{2} - 4 c_{1} c_{3}}}{2c_{1}} = \frac{-c_{2} \pm \sqrt{ 4(c_{1}c_{3} + 1) - 4 c_{1}c_{3}}}{2c_{1}} = \frac{-c_2\pm 2}{2c_1}.
\end{align*} 
Since a non-trivial steady soliton is non-compact, we deduce $c_1=0$ and, by (\ref{est 60a}), $c_2=2$. It is then immediate to solve explicitly for $\alpha$ and $\beta$.  


Next, to the best of our knowledge, the only known explicit solution to the GRS equation \eqref{est 1}, even under the simplifications of $\lambda = q = k =0$, is 
\begin{equation}\label{known explicit}
	H \equiv C_{1}, \hspace{3mm} F \equiv C_{2}, \hspace{3mm} f(t) = C_{3} + C_{4} t 
\end{equation}
for some constants $C_{1}, C_{2}, C_{3}, C_{4} \in \mathbb{R}$. Taking advantage of \eqref{est 60}, we discover a new family of non-trivial explicit solutions to \eqref{est 1} as follows.  

\begin{example}
	Let us consider the transformed system \eqref{est 60} with $q = 0, m = 1,$ and $ k =0$ to obtain 
	\begin{align}\label{est 60 simplified}
		\ddot{\varphi} =& \frac{ \ddot{\beta}}{\beta} - \frac{ (\dot{\beta})^{2}}{2\beta^{2}}, \\
		\ddot{\alpha} =&  - \frac{\dot{\alpha} \dot{\beta}}{\beta} + \dot{\varphi} \dot{\alpha},  \\
		\ddot{\beta} =& - \frac{\dot{\beta} \dot{\alpha}}{\alpha} + \dot{\varphi} \dot{\beta}. 
	\end{align}
	Considering the known explicit solution in \eqref{known explicit}, we can consider the case $F \equiv C$.  Considering $\beta(s) = F^{2}(t)$ from \eqref{est 2}, we see that $\beta \equiv C^{2}$. This leads to, according to \eqref{est 60 simplified}, 
	\begin{align*}
		\ddot{\varphi} = 0 \text{ and } \ddot{\alpha} = \dot{\varphi} \dot{\alpha}.
	\end{align*}
	We can solve the second equation to get $\dot{\varphi} = \frac{\ddot{\alpha}}{\dot{\alpha}} = \partial_{s} \ln ( \lvert \dot{\alpha} \rvert)$. But the first constraint tells us that $\dot{\varphi} = C_{1}$. So, we get $\ln (\lvert \dot{\alpha} \rvert ) = C_{1} s + C_{2}$ and thus $\dot{\alpha} = e^{C_{1}s + C_{2}}$ which leads to $\alpha = \frac{1}{C_{1}} e^{C_{1}s + C_{2}}$. Going back to the functions of $t$, we discover a new family of non-trivial explicit solutions to \eqref{est 1}:  
	\begin{align*}
		F \equiv C, \hspace{3mm} H(t) = \frac{2}{C_{1} (C_{3} - t)}, \hspace{3mm} f (t) = - 2 \ln \left( \frac{ \sqrt{C_{1}} (C_{3} - t)}{2} \right) \text{ for } C_{1} > 0, t < C_{3}; 
	\end{align*}
	several straight-forward generalizations are certainly possible. 
\end{example}

\subsection{Estimates}
Here we assume that $\lambda=0$ and $q=1$ and we'll prove Theorem \ref{main3}. By Theorem \ref{main1}, there must be a singular foil and w.l.o.g., we assume it occurs at $t=0$. The completeness and smoothness of the metric indeed leads to important consequences. First, we observe that the domain of $f$ is $[0, \infty]$; i.e., 
\[|f(t)|<\infty \text{  for any  } t<\infty,\] 
or equivalently via the transformation (\ref{est 2}),
\begin{equation}
	\label{ttoinfity}
	t =\int_0^s\frac{dx}{\sqrt{\alpha(x)}}\rightarrow \infty \text{ as } s\rightarrow \infty;
\end{equation}
A priori either $H$ or $F$ or both must be vanishing at zero. 

\begin{lemma}
	\label{alpha0}
	$H(0)=0$.
\end{lemma}
\begin{proof}
	We proceed by contradiction. Suppose that $H(0)>0$ then $F(0)=0$. That is, the singular foil is a circle $\mathbb{S}^1$. Recall that a complete GRS is real analytic \cite{bando87, dw11}. Henceforth, $F$ and $H$ are real analytic and each has a power series expansion near the singular foil.  
	By (\ref{est 1b}), 
	\[\lambda F^2= \frac{H^2 q^2}{F^2}2m-\frac{H''}{H}F^2-2m \frac{H'F'F}{H}+f'\frac{H'}{H}F^2.\]
	Thus, every term except for $\frac{H^2 q^2}{F^2}2m$ is approaching a finite value as $t\rightarrow 0$, which is a contradiction. 
\end{proof}

\begin{proposition}
	\label{alphamonotone}
	$\alpha$ is monotonic and $\dot{\alpha}(s)>0$ for $s>0$. 
\end{proposition} 
\begin{proof}
	By Lemma \ref{alpha0}, $\alpha(0)=0$ and $\dot{\alpha}(s)>0$ for $s$ near zero. Suppose there is $s_0>0$ such that $\dot{\alpha}=0$. By Lemma \ref{Lemma 4.2}, 
	\[\ddot{\alpha}(s_{0})=\frac{4m\alpha (s_{0}) q^2}{\beta^2 (s_{0})}>0.\]
	Thus, $\alpha(s_0)$ is a local minimum, a contradiction.
\end{proof}

\begin{lemma}
\label{betaatmost}
$\beta$ has at most one local extremum. 
\end{lemma}
\begin{proof}
 At a critical point $s_i$ of $\beta$, $\dot{\beta}=0$ and \eqref{est 60c} implies:
\[\ddot{\beta}(s_i)= \frac{2k}{\alpha (s_{i})}-\frac{4q^2}{\beta (s_{i})}.\]
If $k\leq 0$, then immediately we have $\ddot{\beta}<0$ and the result follows.
	
Suppose that $k>0$ and $\beta$ has more than one local extrema. Then it must have consecutive extrema at $s_1<s_2$. We will then consider all possible cases. First, suppose that $\beta(s_1)$ is a local maximum then 
\[0\geq \ddot{\beta}(s_1)= \frac{2k}{\alpha (s_{1})}-\frac{4q^2}{\beta (s_{1})}.\]
Since $s_1$ and $s_2$ are consecutive extrema, $\dot{\beta}(s)\leq 0$ for all $s\in [s_1, s_2]$; consequently, $\beta$ is non-increasing on that interval. Furthermore, by Lemma \ref{alphamonotone}, $\alpha$ is strictly increasing. Thus,  
\[\ddot{\beta}(s_2) = \frac{2k}{\alpha (s{2})} -  \frac{4q^{2}}{\beta(s_{2})} < \frac{2k}{\alpha (s_{1})} - \frac{4q^{2}}{\beta (s_{1})} <\ddot{\beta}(s_1)<0.\]
Consequently, $\beta(s_2)$ is also a local maximum, a contradiction. The other case follows from an analogous argument.
	
\end{proof}
\begin{lemma}\label{bfiniteablow} If there is a constant $c$ such that $\beta <c$, then $\lim_{s\rightarrow \infty}\alpha(s) =\infty$.
\end{lemma}
\begin{proof}
We proceed by contradiction. Assume that there exists $L \in (0,\infty)$ such that $\alpha(s)< L$ for all $s\in [0, \infty)$. By Lemma \ref{alphamonotone}, 
\begin{align*}
\lim_{s\rightarrow \infty}\alpha(s) &=c_1<\infty,\\
\lim_{s\rightarrow \infty}\dot{\alpha}(s) &=0.
\end{align*}
Consequently, $\liminf_{s\rightarrow \infty} |\ddot{\alpha} (s)|=0$. Also, by Lemma \ref{betaatmost}, 
	\begin{align*}
		\lim_{s\rightarrow \infty}\beta (s) &=c_2<\infty,\\
		\lim_{s\rightarrow \infty}\dot{\beta} (s) &=0.
	\end{align*} 
Next,  (\ref{est 60b}), together with the hypothesis that $q=1$, gives us 
\[\ddot{\alpha} = \frac{\alpha 4m}{\beta^{2}} - \frac{\dot{\alpha} \dot{\beta}}{\beta} + \dot{\varphi} \dot{\alpha}.\]
We have, by (\ref{est 3}), \eqref{est 2}, and (\ref{nablafandS})
\[|\dot{\varphi}|=\frac{|f'|}{\sqrt{\alpha}}<c_3.\]

\textbf{Case 1 $c_2=0$: } By Lemma \ref{betaatmost}, $\dot{\beta}<0$ for all sufficiently large $s$. Also, by assumption we have $\lim_{s\to\infty} \beta(s) = 0$ so that for all $s > 0$ sufficiently large, we have $\beta(s) > 0$; hence, for all such sufficiently large $s > 0$, we have $- \frac{\dot{\alpha}(s) \dot{\beta}(s)}{\beta(s)} > 0$. Consequently, 
\[\ddot{\alpha} = \frac{\alpha  4m}{\beta^{2}} - \frac{\dot{\alpha} \dot{\beta}}{\beta} + \dot{\varphi}\dot{\alpha}> \frac{\alpha 4m}{\beta^{2}} + \dot{\varphi} \dot{\alpha}.\] 
Since $\lim_{s\rightarrow \infty}\dot{\alpha}(s) =0$ and $\lvert \dot{\varphi} \rvert <c_3$, for all $s> s_{1}$ with a sufficiently large $s_1$, we have $\frac{\alpha 4 m}{\beta^{2}} \geq c_{1} 4m$ and $\lvert \dot{\varphi} \dot{\alpha} \rvert < c_{1} 2m$ so that 
\[\ddot{\alpha} \geq \frac{\alpha 4m}{\beta^{2}} - \lvert \dot{\varphi} \dot{\alpha} \rvert > 2m c_1>0.\]
This is a contradiction to  $\liminf_{s\rightarrow \infty} |\ddot{\alpha} (s)|=0$.

\textbf{Case 2 $c_2>0$: } Then
\[\lim_{s\rightarrow \infty} \frac{\dot{\alpha} \dot{\beta}}{\beta}=0.\]
By a similar argument as above, one arrives at a contradiction. 
 
\end{proof}

\begin{lemma}\label{ricatti1}
If $\beta <c$ then for any $\epsilon>0$, there is $s_0(\epsilon)$ such that, for $u=\frac{\dot{\alpha}}{\alpha}$, 
\[ (\dot{u}+u^2+\epsilon u)(s) \geq \frac{4m}{\beta^2(s)}>\frac{4m}{c^2}\]
for all $s\geq s_0$. 
\end{lemma}

\begin{proof}
	If $\beta <c$, 
For $u=\frac{\dot{\alpha}}{\alpha}$, \eqref{est 60b} then reads
\begin{align*}
\dot{u}+u^2 = \frac{4m}{\beta^2}-\frac{m\dot{\alpha} \dot{\beta}}{\alpha \beta} + \dot{\varphi}\frac{\dot{\alpha}}{\alpha} =\frac{4m}{\beta^2}-\frac{m\dot{\beta}}{\beta} u + \dot{\varphi} u.
\end{align*}
By \eqref{est 3c}, \eqref{est 2}, and Lemma \ref{bfiniteablow}, $\lim_{s\rightarrow \infty} \dot{\varphi}(s) = \lim_{s\to\infty} \frac{ f'(t)}{\sqrt{\alpha(s)}} =0$. Thus, by considering two cases on whether $\lim_{s\rightarrow \infty} \beta=0$ or $0<\lim_{s\rightarrow \infty} \beta <\infty$, we have,
for sufficiently large $s_0$,
\[m\frac{\dot{\beta}}{\beta} u - \dot{\varphi} u\leq \epsilon u.\]
The result then follows. 
\end{proof}

\begin{lemma}\label{betablow}
$\lim_{s\rightarrow \infty}\beta(s)=\infty$. 
\end{lemma}
\begin{proof}
	We proceed by contradiction. Suppose that $\beta<c$.  	
	
	\textbf{Claim:} There is sufficiently large $s_1$ and a constant $C>0$ such that, for $u=\frac{\dot{\alpha}}{\alpha}$, $u(s)>C$ for $s\geq s_1$. 
	
	\textbf{Proof of the claim:} For $c_1=\frac{4m}{c^2}$, pick $C=\sqrt{5c_1/11}$ and then $\epsilon=C/10$. By Lemma \ref{ricatti1} there exists a sufficiently large $s_0$ such that for all $s\geq s_0$,  
	\[\dot{u}+u^2+\epsilon u\geq \frac{4m}{\beta^2}>\frac{4m}{c^2}=c_1.\]
	If $u(s)\leq C$ for all $s \geq s_{0}$, then, since $0<u$ by Lemma \ref{alphamonotone}, $\lvert u \rvert \leq C$ implying that 
	\[\dot{u} \geq  c_1-(C^2+\epsilon C) =\frac{c_1}{2}. \] 
	Thus, there is $s_1 (s_0, c_1)>s_0$ such that $u(s_1)>C$. Then, by the maximum principle, $u(s)>C$ for all $s\geq s_1$. The proof of the claim is finished. 
	
Next, we integrate $u(s)$ over $[s_{1}, s_{1}+ \ell]$, use the fact that $u(s) > C$ to deduce 
\begin{align*} C \ell \leq \int_{s_1}^{s_1+\ell} u ds=  \int_{s_1}^{s_1+\ell} \frac{\dot{\alpha}}{\alpha} ds \leq \ln(\alpha)(s_1+\ell)-\ln(\alpha)(s_1).
\end{align*}	
Therefore, $\alpha$ has at least exponential growth, which is a contradiction to (\ref{ttoinfity}).  
\end{proof}

\begin{proof}[Proof of Theorem \ref{main3}] It follows from Lemmas \ref{alphamonotone}, \ref{betaatmost}, and \ref{betablow} in view of (\ref{est 3}).
	
\end{proof}

\section*{Acknowledgements}
This material is based upon work supported by the National Science Foundation under Grant No. DMS-1928930, while the first author was in residence at the MSRI in Berkeley, California, during the Fall semester of 2024. The second author is supported by the Simons Foundation (962572, KY). We also would like to thanks VIASM for its hospitality and support during our visit in summer 2025. 

\def\cprime{$'$}
\bibliographystyle{plain}
\bibliography{bioMorse}

\end{document}